\documentclass{ws-procs9x6}

\usepackage{amssymb, graphicx}
\usepackage{color}
\usepackage[all]{xy}
\usepackage{comment}
\usepackage{graphicx}
\usepackage{graphics}
\usepackage{enumerate}
\usepackage{xypic}
\usepackage{hyperref}
\xyoption{all}

\newcommand{\LR}{\Leftrightarrow}
\newcommand{\lra}{\leftrightarrow}

\newcommand{\ttext}[1]{\ \text{#1}\ }

 \newcommand{\leT}{\le_T}
\newcommand{\geT}{\ge_T}

\newcommand{\fa}{\forall}
\newcommand{\frb}{\mathfrak{b}}
\newcommand{\frd}{\mathfrak{d}}
\newcommand{\frs}{\mathfrak{s}}
\newcommand{\frr}{\mathfrak{r}}
\newcommand{\calN}{\mathcal{N}}
\newcommand{\calS}{\mathcal{S}}
\newcommand{\calC}{\mathcal{C}}
\newcommand{\calE}{\mathcal{E}}
\newcommand{\calI}{\mathcal{I}}
\newcommand{\calJ}{\mathcal{J}}
\newcommand{\R}{\mathbb{R}}
\newcommand{\calF}{\mathcal{F}}
\newcommand{\calM}{\mathcal{M}}
\newcommand{\frv}{\mathfrak{v}}
\newcommand{\frw}{\mathfrak{w}}
\newcommand{\calB}{\mathcal{B}}
\newcommand{\calD}{\mathcal{D}}
\newcommand{\calG}{\mathcal{G}}

\newcommand\+[1]{\mathcal{#1}}

\newcommand{\se}{\mathfrak{se}}

\newcommand{\sN}[1]{_{#1\in \omega}}
\newcommand{\sub}{\subseteq}

\DeclareMathOperator{\cov}{cover}
\DeclareMathOperator{\non}{non}
\DeclareMathOperator{\cof}{cofin}
\DeclareMathOperator{\add}{add}

\DeclareMathOperator{\dom}{dom}

\newcommand{\Baire}{{}^\omega\omega}
\newcommand{\Cantor}{{}^\omega2 }

\begin{document}

\title{An analogy between cardinal characteristics   and   highness properties of oracles}

\author{J\"org Brendle}
\address{Graduate School of System Informatics\\
Kobe University\\
Rokko-dai 1-1\\
Nada, Kobe, 657-8501\\
Japan\\
Email: brendle@kurt.scitec.kobe-u.ac.jp}

\author{Andrew Brooke-Taylor}
\address{School of Mathematical Sciences\\
University of Bristol\\
University Walk\\
Bristol, BS8 1TW\\
United Kingdom\\
Email: a.brooke-taylor@bristol.ac.uk}

\author{Keng Meng Ng}
\address{Division of Mathematical Sciences\\ School of Physical \& Mathematical Sciences\\
Nanyang Technological University\\ 21 Nanyang Link\\ Singapore\\
Email: selwyn.km.ng@gmail.com}

\author{Andr\'e Nies}
\address{Department of Computer Science\\ Private Bag 92019 \\
University of Auckland\\
Auckland\\
New Zealand\\
Email: andre@cs.auckland.ac.nz\\}

\let\thefootnote\relax\footnote{J\"org Brendle is supported by Grant-in-Aid for Scientific Research (C) 24540126,
   Japan Society for the Promotion of Science.
Andrew Brooke-Taylor is currently supported by a UK
Engineering and Physical Sciences Research Council Early Career Fellowship,
and was previously supported by a JSPS Postdoctoral Fellowship for Foreign Researchers and JSPS Grant-in-Aid 23 01765.
Keng Meng Ng is supported by the MOE grant MOE2011-T2-1-071. Andr\'e Nies is supported by the Marsden fund of New Zealand. }

\begin{abstract} We present an analogy between  cardinal characteristics from set theory and  highness properties from computability theory, which specify a sense in which   a Turing oracle is computationally strong.
While this analogy was first studied explicitly by Rupprecht  ({\em Effective correspondents to cardinal characteristics in
  {C}icho\'n's diagram},  PhD thesis, University of Michigan, 2010), many  prior  results   can be viewed from  this
perspective.
After a comprehensive survey of the analogy
for characteristics from Cicho\'n's diagram,
we extend it to
  Kurtz randomness and the analogue of the Specker-Eda number.
\end{abstract}

%
%

\section{Introduction}

Mathematics studies abstract objects via concepts and corresponding methods. Metamathematics studies these concepts and methods. A common scheme in metamathematics is duality, which can be seen as a bijection between concepts. For instance, Stone duality matches concepts related to Boolean algebras with   concepts related to    totally disconnected Hausdorff spaces. A weaker scheme is analogy, where different areas develop in similar ways. An example is the analogy between meager sets and null sets. While one   can match the  concepts  in one area with the concepts in the other area, the results about them  may differ.

We systematically develop an analogy between \begin{itemize}\item[(a)]  cardinal characteristics from set theory, which broadly speaking measure the deviation from the continuum hypothesis of a particular  model of ZFC
\item[(b)] highness properties from computability theory, which specify a sense in which   a Turing oracle is computationally strong.  \end{itemize}

One of the simplest examples of a cardinal characteristic is the unbounding number.  For functions $f, g$ in Baire space $\Baire$, let $f \le^* g$ denote that $f(n) \le g(n)$ for almost all $n \in \omega$. Given a countable collection  of functions $(f_i)\sN i$ there is $g$ such that $f_i \le^* g$ for each $i$: let $g(n) = \max_{i \le n} f_i(n)$.  How large a collection of functions do we need so that no  upper bound $g$ exists? The unbounding number $\frb$ is the least size of such a collection of functions; clearly $\aleph_0 < \frb \le 2^{\aleph_0}$.

In computability theory,  probably the simplest highness property is the following: an oracle $A$ is called (classically) high  if $\emptyset '' \le_T A'$. Martin \cite{Martin:66} proved that one can require  equivalently the following:   there is a function $f \le_T A$ such that $g \le^* f$ for each computable function $g$.

At the core of the analogy, we will describe a formalism to transform cardinal characteristics into highness properties.   A ZFC provable relation $\kappa \le \lambda$ between cardinal characteristics turns into a containment: the highness property for $\kappa$ implies the one for~$\lambda$.

The analogy  occurred  implicitly in the work of Terwijn and Zambella \cite{Terwijn.Zambella:01}, who showed that being  low for Schnorr tests is equivalent to  being computably traceable. (These are  lowness properties, saying the oracle is close to being computable; we obtain  highness properties by taking   complements.)  This is the computability theoretic analog of a result of Bartoszy\'nski~\cite{Bartoszynski:84} that the cofinality of the null sets (how many null sets does one need to cover all null sets?) equals the domination number for traceability, which we will later on denote $\frd(\in^*)$.  Terwijn and Zambella alluded to some connections with  set theory in their work. However, it was not Bartoszy\'nski's work, but rather   work on rapid filters by Raisonnier~\cite{Raisonnier:84}. Actually,  their proof bears striking similarity to Bartoszy\'nski's; for instance, both proofs use measure-theoretic calculations involving independence.  See also the books Ref.~\citenum{Bartoszynski.Judah:book}, Section~2.3.9, 
and Ref.~\citenum{Nies:book}, Section~8.3.3.

The analogy was first observed and studied explicitly by Rupprecht
\cite{Rup:The, Rupprecht:10}. Let $\add(\calN)$ denote the additivity of the null sets: how many null sets does one need so that their union is not null?  Rupprecht found the computability-theoretic analog of $\add(\calN)$.  He called this highness property ``Schnorr covering''; we prefer to call it ``Schnorr engulfing''. A Schnorr null set is a certain effectively defined kind of null set. An oracle $A$ is Schnorr engulfing if it computes a Schnorr null set that contains all plain Schnorr null sets.  While $\add(\calN)$ can be less than  $\frb$, Rupprecht showed that the Schnorr engulfing  sets are exactly the high sets. Thus, we only have an analogy, not full duality.

%
%

 \section{Cardinal characteristics and Cicho\'n's diagram} \label{s:Cichon}
All our cardinal characteristics will be given as the unbounding and  domination numbers of suitable relations.
Let  $R \subseteq  X \times Y$ be a relation between spaces $X,Y$ (such as Baire space) satisfying $\forall x \; \exists y \;
(x R y)$ and $\forall y \; \exists x \; \neg (x R y)$. Let $S = \{  \langle y,x \rangle \in Y \times X \colon \neg xR y\}$.   We write

\[ \frd(R) = \min\{|G|:G\subseteq Y \land \, \forall x \in X \,
\exists y \in  G   \, xR y\}.\]

\[ \frb(R) = \frd (S) =  \min\{|F|:F\subseteq X \land \, \forall y\in Y
\exists x \in F   \, \neg  xR y\}.\]
$\frd(R)$ is called the \emph{domination number} of $R$, and $\frb(R)$ the \emph{unbounding  number}.

If $R$ is a preordering without greatest element, then   ZFC proves $\frb (R) \le \frd(R)$.  To see this,   we show that  any dominating set $G$ as in the definition of $\frd(R)$ is an unbounded set as in the definition of $\frb(R)$. Given  $y$     take
a $z$  such that  $\neg zRy$. Pick
    $x\in G$ with $zRx$. Then  $\neg xRy$.

For example, the  relation  $\le^*$ on $ \Baire \times \Baire$ is a preordering without maximum.  One often writes $\frb $ for $\frb (\le^*)$  and  $\frd $ for $\frd (\le^*)$;   thus  $\frb \le \frd$.
(Another easy exercise is to show that if $R$ is a preordering, then $\frb(R)$ is a regular cardinal.)

\subsection{Null sets and meager sets} \label{ss:null_meager} Let $\calS  \sub \mathcal P(\R)$ be  a collection of ``small'' sets; in particular, 
assume $\calS$ is closed downward under inclusion, each singleton  set is in $\calS$, and  $\R$ is not in $\calS$. We will mainly consider the case when $\calS$ is  the class of  null sets  or the class of  meager sets. For null or meager sets, we can replace $\R$ by Cantor space or Baire space without changing the cardinals.

 The unbounding and the domination number for the subset relation $\sub_\calS$ on $\calS$  are called additivity and cofinality, respectively. They  have special notations:
\begin{eqnarray*} \add (\calS) & = & \frb (\sub_\calS) \\
	\cof (\calS) & = & \frd (\sub_\calS) \end{eqnarray*}
Let $\in_\calS$ be  the membership relation  on $\R \times \calS$. The unbounding and domination numbers for membership   also have special notations:
	\begin{eqnarray*}
 \non(\calS) &=& \frb(\in_\calS) =  \min\{|U|:U\subseteq\R\land U\notin\calS\} \\
	\cov(\calS) &=& \frd(\in_\calS) =  \min\{|\calF|:\calF \subseteq\calS\land\bigcup\calF=\R\}
\end{eqnarray*}
The   diagram in Fig.\ \ref{Fig:4-diagram} shows the ZFC relationships between these cardinals.  An arrow $\kappa \to \lambda$ means that ZFC proves $\kappa \le \lambda$.
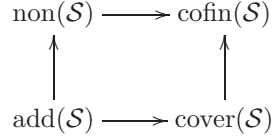
\begin{figure}[htbp]
\begin{equation*} \xymatrix{
 \non(\calS)  \ar[r]&	\cof (\calS)   \\
	\add (\calS) \ar[u]\ar[r]&	\cov (\calS) \ar[u]
} \end{equation*}
\caption{Basic ZFC relations of  characteristics for a class $\calS$.}
\label{Fig:4-diagram}
\end{figure}
The only slightly nontrivial arrow is  $\non(\calS)  \to 	\cof (\calS)$:  Suppose we are given $\calF \sub \calS$ such that  for every $C \in \calS$ there is $D \in \calF$ with $C \sub D$. Using the axiom of choice, for each $D$ pick $u_D \not \in D$.  Now let $V = \{ u_D \colon \, D \in \calF\}$. Then $V \not \in \calS$ and $|V| \le |\calF|$.
(Note that we have used the notations ``$\cov$" and ``$\cof$" instead of the standard         ``cov" and``cof", because the latter two look very much alike.)

\subsection{The combinatorial Cicho\'n diagram}  \label{ss:combCichon}   In a somewhat nonstandard approach to Cicho\'n's diagram, we will  consider the   smaller
  ``combinatorial'' diagram in Fig.\ \ref{Fig:6-diagram} which describes the ZFC relations between   the      cardinal characteristics $\frd(R)$ and $\frb(R)$ for three relations $R$. The first relation is $\le^*$.
The second  relation   is   $$\{  \langle f, g \rangle \in \Baire \times \Baire \colon \, \forall^\infty n  \, f(n) \neq g(n)\},$$ which we will denote by~$\neq^* $.  For instance, we have
\[ \frd(\neq^*)
=\min\{|F|:F\subseteq\Baire\land\forall e\in\Baire
\exists f\in F\forall^\infty n\in\omega(e(n)\neq f(n))\} \]

For the third relation, let $Y$ be the space of functions from $\omega$ to the set of finite subsets of $\omega$. Recall that  $\sigma \in Y$ is a \emph{slalom} if $|\sigma(n)| \le n$ for each $n$.   We say that  a function $f \in \Baire$ is \emph{traced} by $\sigma$ if $f(n) \in \sigma(n)$ for almost every $n$. We denote this relation on $\Baire \times Y$ by $\in^*$.
We have for example
\[
\frb(\in^*)
=\min\{|F|:F\subseteq\Baire\land\forall\text{ slalom }\sigma
\exists f\in F\exists^\infty n\in\omega(f(n)\notin \sigma(n))\}.
\]

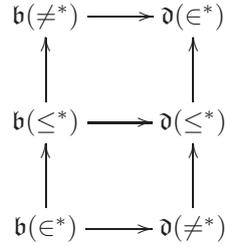
\begin{figure}[htbp]
	\begin{equation*}
	\xymatrix{
	\frb(\neq^*) \ar[r]&\frd(\in^*)\\
	\frb(\le^*)\ar[u]\ar[r]&\frd(\le^*)\ar[u] \\
	\frb(\in^*)\ar[u]\ar[r]&\frd(\neq^*)\ar[u]
	}
	\end{equation*}
\caption{Combinatorial ZFC relations.}
\label{Fig:6-diagram}
\end{figure}

The nontrivial arrows such as $\frb(\in^*) \to \frd(\neq^*)$ follow from the full diagram discussed next.

\subsection{The  full Cicho\'n diagram}

We  are now ready to present (a slight extension of) the usual Cicho\'n diagram. As a first step, in Fig.~\ref{Fig:4-diagram} we take the equivalent diagram for $\calS= \calN$ where $\cov(\calN)$ and $\non(\calN)$ have been interchanged. We join it with the diagram for  $\calS= \calM$ and obtain the diagram in Fig.~\ref{Fig:8-diagram}.
\begin{figure}[htbp]
\begin{equation*}
\xymatrix{
\cov(\calN)\ar[r]&\non(\calM)\ar[r]&\cof(\calM)\ar[r]&\cof(\calN)\\
\add(\calN)\ar[u]\ar[r]&\add(\calM)\ar[u]\ar[r]&\cov(\calM)\ar[u]\ar[r]&
\non(\calN)\ar[u]
}
\end{equation*}
\caption{The diagrams for $\calN $ and $\calM$ joined.}
\label{Fig:8-diagram}
\end{figure}
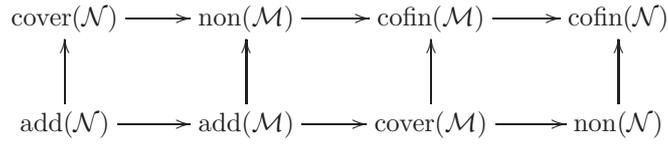
The new arrows joining the two 4-element diagrams, such as $\cof(\calM) \to \cof(\calN)$, are due to Rothberger and Bartoszy\'nski; see
Ref.~\citenum{Bartoszynski.Judah:book} Sections~2.1.7 and 2.3.1 for details.

  Finally in Fig.\ \ref{Fig:10-diagram} we superimpose  this 8-element diagram with  the combinatorial 6-element diagram in Fig.\ \ref{Fig:6-diagram}.  For all its  elements except  $\frb$ and $\frd$,  ZFC    proves  equality  with  one  of the characteristics in the 8-element diagram.   These
four  ZFC equalities are due to Bartoszy\'nski~\cite{Bartoszynski:84,Bartoszynski:87} and Miller~\cite{Miller:82}. Below we will mainly rely on the  book Ref.~\citenum{Bartoszynski.Judah:book} Sections~2.3 and 2.4.
\begin{figure}[htbp]
\begin{equation*}
\xymatrix{
&\frb(\neq^*)\ar@{=}^{\text{Ref.~\citenum{Bartoszynski.Judah:book} Th.~2.4.7}}[d]&&\frd(\in^*)\ar@{=}^{\text{Ref.~\citenum{Bartoszynski.Judah:book} Th.~2.3.9}}[d]\\
\cov(\calN)\ar[r]&\non(\calM)\ar[r]&\cof(\calM)\ar[r]&\cof(\calN)\\
&\frb\ar[u]\ar[r]&\frd\ar[u]&\\
\add(\calN)\ar[uu]\ar [r]&\add(\calM)\ar[u]\ar[r]&\cov(\calM)\ar[u]\ar[r]&
\non(\calN)\ar[uu]\\
\frb(\in^*)\ar@{=}^{\text{Ref.~\citenum{Bartoszynski.Judah:book} Th.~2.3.9}}[u]&&\frd(\neq^*)\ar@{=}^{\text{Ref.~\citenum{Bartoszynski.Judah:book} Th.~2.4.1}}[u]
}
\end{equation*}
\caption{Cicho\'n's  diagram.}
\label{Fig:10-diagram}
\end{figure}
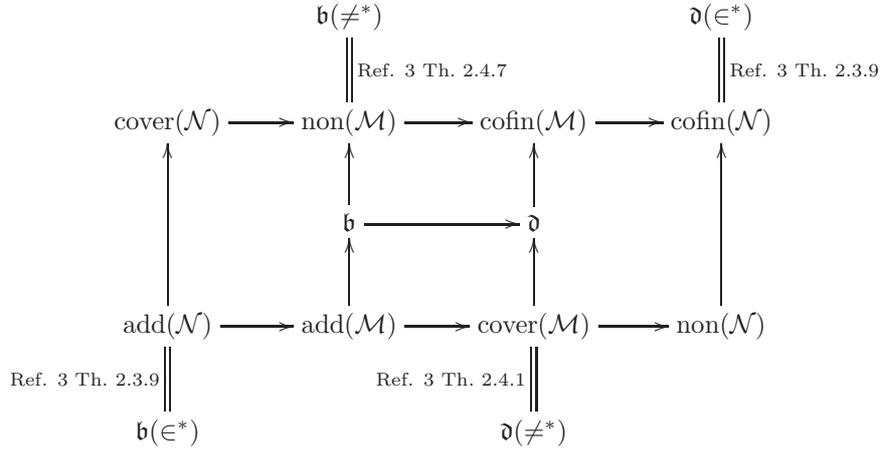

Fig.\ \ref{Fig:10-diagram}
shows \emph{all}   ZFC-provable binary relationships: for any two cardinal
characteristics $\frv$ and $\frw$ appearing in the diagram, if there is no arrow
$\frv\to \frw$, then there is a known model of ZFC in which $\frv>\frw$. See  Subsection~\ref{ss:separate} for more discussion related to this.

Two  ZFC-provable ternary relationships  will be
of interest to us:
\begin{align*}
\add(\calM)&=\min(\frb,\cov(\calM)),  & \text{Miller, Truss;  see Ref.~\citenum{Bartoszynski.Judah:book} Section~2.2.9}\\
\cof(\calM)&=\max(\frd,\non(\calM))  & \text{Fremlin; see Ref.~\citenum{Bartoszynski.Judah:book} Section~2.2.11}.\\
\end{align*}

%
%

\section{Highness properties corresponding to the cardinal characteristics}

\label{s:highness_props}
We now present a scheme to describe highness properties of oracles which is analogous to the one in Section~\ref{s:Cichon}. As before, let  $R \subseteq  X \times Y$ be a relation between spaces $X,Y$, and  let $S = \{  \langle y,x \rangle \in Y \times X \colon \neg xR y\}$.
Suppose we have specified what it means for     objects $x$ in $X$,    $y$ in  $Y$  to be computable in a Turing oracle $A$. We denote this by for example $x\leT A$. In particular, for $A= \emptyset$ we have a notion of computable objects.
For instance, if $X$ is Baire space and $f\in X$, we have the usual notion $f \le_T A$.

Let the variable $x$ range over $X$, and let $y$ range over $Y$. We define the highness properties

\[ \calB(R) =   \{ A:  \,\exists y \leT A \, \fa x \ttext{computable} [xRy]     \} \]

\[ \calD(R) =  \calB(S) =   \{ A:  \,\exists x \leT A \, \fa y \ttext{computable} [\neg xRy]   \}  \]
If $R$ is a preordering with no greatest  computable element, then clearly  $\calB (R) \sub \calD(R)$. We will give some examples of such preorderings in Subsection~\ref{ss:comb_rel}.

Comparing these definitions with the ones of $\frd(R)$ and $\frb(R)$ at the beginning of Section~\ref{s:Cichon},
one notes that, ignoring the domains of quantification,
we use here direct analogs of the
\emph{negations} of the statements there.
For example, the formula defining $\frb(\leq^*)$ is of the form
$\forall y \exists x\,\lnot(x\leq^* y)$, and
the defining formula for $\calB(\leq^*)$ takes the form of its negation,
$\exists y \forall x (x\leq^* y)$.
The main reason for doing this is the connection to forcing.
 The cardinal characteristics
we consider are each defined as
\(
\min\{|\calF|:\phi(\calF)\}
\)
for some property $\phi$, where $\calF$ is a set of
functions $\omega\to\omega$, or meager sets, or Lebesgue null sets.
In each case, there is a forcing that introduces via a generic object a function
(or meager set, or null set) such that in the extension model,
$\phi$ no longer holds of the set of all ground model functions
(respectively, meager sets, null sets).
In the $\frb$ ($=\frb(\leq^*)$) case, for instance, after adding a function $y_0$ that dominates all functions
from the ground model, the defining formula $\forall y \; \exists x \; \neg (x \leq^* y)$ no longer holds for the
ground model functions~$x$, as witnessed by $y = y_0$. Thus, iterating this procedure increases the value of $\frb$.
Building a generic object is analogous to building an oracle that is computationally powerful in the sense specified by the analog of $\phi$.

A notational  advantage of taking the negations is that, as the analog of $\frb$, we obtain  classical highness, rather than  the lowness property of being non-high.

\subsection{Schnorr null sets and effectively meager sets}  \label{ss:Schnorr meager}   We   find effective versions of null and meager sets. For null sets, instead of in  $\R$ we will   work in Cantor space $\Cantor$. Let $\lambda$ denote the usual product measure on $\Cantor$. For meager sets we will work in Cantor space,  or sometimes  in Baire space $\Baire$.

A {\em Schnorr  test} is an effective sequence $(G_m)\sN m$ of  $\Sigma^0_1$ sets such that each $G_m$ has measure $\lambda G_m$ less than or equal to $2^{-m}$ and this measure is a computable real uniformly in $m$.  A set   $\calF \sub \Cantor$ is called \emph{Schnorr  null} if  $\calF \sub \bigcap_m G_m$ for a Schnorr test $(G_m)\sN m$.

An {\em effective $F_\sigma$ class} has the form $  \bigcup_m \+ C_m$, where the $\+ C_m$ are uniformly $\Pi^0_1$. A set   $\calF \sub \Cantor$ is called \emph{effectively meager}    if it is contained in such a class $  \bigcup_m \+ C_m$ where  each $\+ C_m$ is nowhere dense. (In this case, from $m$ and a string $\sigma$ we can compute a string $\rho \succeq \sigma$ with $[\rho] \cap \+ C_m= \emptyset$. Informally, for a $\Pi^0_1$ class,  being nowhere dense is effective by nature.)

We now obtain  $4+4$  highness properties according to the relations specified in Subsection~\ref{ss:null_meager}.

\subsubsection{Effectively meager sets} \label{ss:EM}

\begin{equation*}
\xymatrix{
 \calB(\in_\calM)  \ar^{\text{(3)}}[r]&	\calD (\sub_\calM)  \\
	\calB (\sub_\calM) \ar^{\text{(2)}}[u]\ar^{\text{(1)}}[r]&\calD(\in_\calM) \ar^{\text{(4)}}[u]
}
\end{equation*}

We clarify   the  meaning    of   these  properties  of an oracle $A$, and introduce some terminology or match them with known notions in computability theory.

\begin{itemize} \item
$\calB (\sub_\calM)$:     there is an $A$-effectively meager set $\calS$ that includes  all effectively meager sets.   Such an   oracle $A$ will be called \emph{meager engulfing}.

\item
$\calB(\in_\calM)$:  there is an $A$-effectively meager set that contains  all computable reals. Such an   oracle $A$   will be called  \emph{weakly meager engulfing}. Note that the notion of (weakly) meager engulfing is the same in the Cantor space and in the Baire space.

\item
$\calD (\sub_\calM) $:  there is an $A$-effectively meager set not included   in any effectively meager set. It is easy to see  that this is the same as saying that $A$ is not low for weak 1-genericity (see Ref.~\citenum{Stephan.Yu:nd} Theorem~3.1).

\item
$\calD(\in_\calM)$: there is $f \leT A$ such that $f \not \in \calF$ for each effectively meager $\calF$. This says that $A$ computes a weakly 1-generic.  \end{itemize}

An arrow now means containment. The various arrows can be  checked easily.
\begin{itemize} \item[(1)]  Given an $A$-effectively meager set $\calS \sub \Cantor$, by finite extensions  build $f \in \Cantor$, $f \leT A$ such that $f \not \in \calS$. If $\calS$ includes all  effectively meager sets,  then  $f$ is not in  any nowhere dense $\Pi^0_1$ class, so $f$ is weakly 1-generic.
	\item[(2)] Trivial.
	
	\item[(3)]  Let $\calF$ be an  $A$-effectively meager set containing all computable reals. If $\calF $ is included in an effectively meager set $\calG$, then choose a computable $P \not \in \calG$ for a contradiction.
	\item[(4)] Trivial.
\end{itemize}

\subsubsection{Schnorr null sets}
\label{ss:ES}

In order to join diagrams later on, for measure we work with the equivalent flipped diagram from  Subsection~\ref{ss:null_meager}  where the left upper and right lower corner have been exchanged.
\begin{equation*}
\xymatrix{
 \calD(\in_\calN)  \ar[r]&	\calD (\sub_\calN)  \\
	\calB (\sub_\calN) \ar[u]\ar[r]&\calB(\in_\calN) \ar[u]
}
\end{equation*}

\begin{itemize} \item
$\calB (\sub_\calN)$:     there is a  Schnorr null  in $A$ set that includes  all Schnorr null sets.   Such an   oracle $A$ will be called \emph{Schnorr engulfing}.  (This was     called ``Schnorr covering'' by Rupprecht \cite{Rup:The, Rupprecht:10}.  We have changed his terminology because the computability theoretic class is not the analog of a   cardinal characteristic of type $\cov(\+ C)$.)

\item
$\calB(\in_\calN)$:  there is a Schnorr null  in $A$ set that contains  all computable reals. Such an   oracle $A$   will be called  \emph{weakly Schnorr engulfing}.

\item
$\calD (\sub_\calN) $:  there is a Schnorr null  in $A$ set not contained in any Schnorr null sets. One says  that $A$ is not low for Schnorr tests following Ref.~\citenum{Terwijn.Zambella:01}.

\item
$\calD(\in_\calN)$: there is $x \leT A$ such that $x \not \in \calF$ for each Schnorr null $\calF$. This says that $A$ computes a Schnorr random.
\end{itemize}
As before, the arrows are easily verified. One uses the well-known fact that each Schnorr null set fails to contain some computable real; see for example  Ref.~\citenum{Nies:book} Ex.~1.9.21 and the solution.	 This was already observed by Rupprecht~\cite{Rupprecht:10}.

\subsection{Combinatorial relations} \label{ss:comb_rel}
Let us see which highness properties we obtain for the three relations in Subsection~\ref{ss:combCichon}.

\begin{itemize}  \item If $R$ is $\le^*$, then $\calB(R)$ is highness, and $\calD(R)$ says that an oracle $A$  is of hyperimmune degree.
\item
Let $R$ be $\neq^*$.  The property  $\calB(\neq^*)$ says that there is  $f \leT A $ such that  $f$   eventually disagrees  with  each computable function.
Recall that a set   $A$ is called {\em DNR} if it computes a function $g$
that is \emph{diagonally nonrecursive}, i.e.,
there is no $e$ such that the $e$th partial computable function converges
on input $e$ with output $g(e)$
(this is also referred to as diagonally noncomputable or d.n.c., for example
in Ref.~\citenum{Nies:book}). $A$ is called {\em PA} if it computes a $\{ 0,1 \}$-valued diagonally nonrecursive function.
The property $\calB(\neq^*)$ is equivalent to
  ``high or DNR'' by Theorem~5.1 of Ref.~\citenum{Kjos.ea:2011}.

The property  $\calD(\neq^*)$ says that there is  $f \leT A $ such that  $f$   agrees infinitely often   with  each computable function.

\item
Let $R$ be $\in^*$. Slaloms are usually called \emph{traces} in computability theory. Recall that $D_n$ is the $n$-th finite set. We say that a trace $\sigma$ is \emph{computable in $A$}  if there is a  function $p \leT A$ such that $\sigma(n) = D_{p(n)}$.   Now $\calB(\in^*)$ says that $A$ computes a trace that traces all computable functions. By Theorem 6 of Ref.~\citenum{Rupprecht:10} this is equivalent to highness.

 The property $\calD(\in^*)$ says that there is a function $f \leT A$ such that, for each computable trace $\sigma$, $f$ is not traced by $\sigma$. This means that $A$ is not computably traceable in the sense of Ref.~\citenum{Terwijn.Zambella:01}.
\end{itemize}

%
%

\section{The full  diagram in computability theory}

We now present the full analog of Cicho\'n's diagram.
Note that in Theorem~IV.7 of his thesis \cite{Rup:The} Rupprecht also gave this diagram
for the standard part of Cicho\'n's diagram,
without the analogs of $\frb(\in^*)$, $\frd(\in^*)$, $\frb(\neq^*)$, and
$\frd(\neq^*)$ explicitly mentioned. Most of the equivalences between the analog of the standard Cicho\'n diagram and its full  form are implicit in the literature; see Fig.\ \ref{Fig:computability_diagram}.

The bijection between  concepts in set theory and in computability theory is obtained as follows. In Section~\ref{s:highness_props} we have already specified effective versions of each of the relations $R$ introduced in Section~\ref{s:Cichon}. In Cicho\'n's diagram of Fig.\ \ref{Fig:10-diagram}, express each characteristic in the original form   $\frb(R)$ or $\frd(R)$, and  replace it by $\calB(R)$ or $\calD(R)$, respectively.  Replacing most of these notations by their meanings defined and explained  in Section~\ref{s:highness_props}, we obtain the diagram of Fig.\ \ref{Fig:computability_diagram}.

\begin{figure}[htbp]
\begin{equation*}
\xymatrix{
&  \text{high or DNR} \lra \calB(\neq^*)
\ar^{\parbox{1cm}{\scriptsize \text{Ref.~\citenum{Rup:The}}\\ \text{(Thm. \ref{thm:WME})}}}@{=}[d]
&&\parbox{2.5cm}{\center $\calD(\in^*) \lra \  $not\\ computably traceable}\ar^{\text{Ref.~\citenum{Terwijn.Zambella:01}}}@{=}[d]\\
\parbox{1.8cm}{\center $A\ge_Ta$  Schnorr random}\ar^{\text{Ref.~\citenum{Nies.Stephan.ea:05}}}[r]&
\parbox{2cm}{\center{weakly meager engulfing}}\ar[r]&
\parbox{2cm}{\center not low for weak 1-gen  \\  (i.e. hyperimmune or DNR\cite{Stephan.Yu:nd})}\ar^{\text{Ref.~\citenum{Kjos.ea:2005}}}[r]&
\parbox{1.8cm}{\center not low for Schnorr tests}\\
&\text{high}\ar[u]\ar[r]&
\parbox{2.3cm}{\center hyperimmune degree}\ar[u(0.45)]&\\
\parbox{1.3cm}{\center Schnorr engulfing} \ar[uu(0.85)]\ar@{=}[r]
\ar^{\text{Ref.~\citenum{Rupprecht:10}}}@{=}[ur]&
\parbox{1.3cm}{\center meager engulfing}\ar_{
\parbox{1cm}{\scriptsize \text{Ref.~\citenum{Rup:The}}\\ \text{(Thm. \ref{thm:highmeager})}}}@{=}[u]
\ar[r]&
\parbox{1.4cm}{\center weakly\\ 1-generic degree}\ar^{\text{Ref.~\citenum{Kurtz:81}}}@{=}[u]\ar^{\text{Ref.~\citenum{Rupprecht:10}}}[r]&
\parbox{1.3cm}{\center weakly Schnorr engulfing}\ar[uu(0.85)]\\
 \calB(\in^*)
\ar^{\text{Ref.~\citenum{Rupprecht:10}}}@{=}[u]&&
 \calD(\neq^*)\ar^{\text{Prop. \ref{prop:W1G}}}@{=}[u]
}
\end{equation*}
\caption{The analog  of Cicho\'n's  diagram in computability.}
\label{Fig:computability_diagram}
\end{figure}
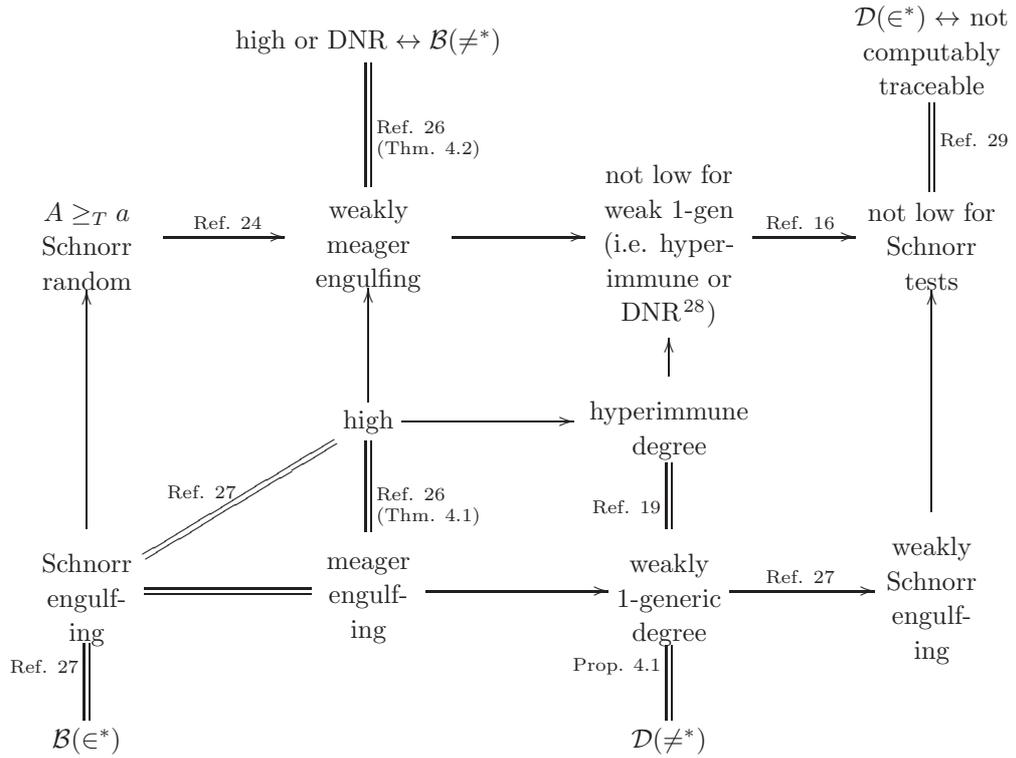

Note that there is a lot of collapsing: instead of ten distinct  nodes, we have only seven. Recall that we rely on a  specific  way of effectivizing the relations $R$.   T.\ Kihara has raised the question whether, if one instead  chooses an effectivization via higher computability theory,   there is less  collapsing. He has announced that in the hyperarithmetical  case, the analogs of $\frd$ and $\cov(\+ M)$ do not coincide.

As  the analog of  $\cof(\calM)=\max(\frd,\non(\calM))$, we have that

\begin{center}   not low for weak 1-genericity   $=$ weakly meager engulfing $\cup$ h.i.\ degree. \end{center}  This is so by the degree theoretic characterizations and because highness  implies being of  hyperimmune degree.   The analog of the dual  ternary relation    $\add(\calM)=\min(\frb,\cov(\calM))$   is trivial because of the collapsing.


\subsection{Implications}

We verify the arrows and equalities in the computability version of Cicho\'n's diagram Fig.\ \ref{Fig:computability_diagram} in case they 
are not the trivial  arrows from Subsections~\ref{ss:EM} and \ref{ss:ES},
and they have not been referenced in the diagram or have only appeared in Rupprecht's thesis (Ref.~\citenum{Rup:The}). This only leaves   highness properties relating to meagerness.  

As indicated in the diagram, Kurtz~\cite{Kurtz:81} has shown that the weakly 1-generic and hyperimmune degrees coincide (also see Ref.~\citenum{Downey.Hirschfeldt:book} Corollary~2.24.19).
\begin{proposition}\label{prop:W1G}  $A$ is in $\calD(\neq^*)$  $\LR$   $A$ has weakly 1-generic degree.
\end{proposition}
\begin{proof}
($\Rightarrow$): Suppose  $f\leT A$   infinitely often  agrees with  each computable function. Then $f+1$ is not dominated by any computable function,   so $A$ is of hyperimmune degree.

($\Leftarrow$): Suppose that $A$ is weakly 1-generic. Let $f(n)= $ least $i\geq 0$ such that $n+i\in A$. Let $h$ be a total computable function. It is enough to argue that $f(x)=h(x)$ for some $x$. We let $V=\{\sigma*0^{h\left(|\sigma|\right)}*1\mid\sigma\in2^{<\omega},|\sigma|>0\}$. Clearly $V$ is a dense computable set of strings. Let $\sigma$ be such that $A\supset \sigma*0^{h(|\sigma|)}*1$. Then $f(|\sigma|)=h(|\sigma|)$.
\end{proof}

Rupprecht (Ref.~\citenum{Rup:The} Corollary~V.46) showed that meager engulfing is
equivalent to high.  Our proof below goes by way of a further
intermediate characterization.

\begin{theorem}\label{thm:highmeager} The following are equivalent for an oracle $A$.
\begin{itemize}
\item[(i)] $A$ is high.
\item[(ii)] There is an effective relative to $A$  sequence $\{G_k\}$ of nowhere dense $\Pi^0_1(A)$-classes such that each nowhere dense $\Pi^0_1$ class  equals  some $G_k$.
\item[(iii)] $A$ is meager engulfing. \end{itemize}
\end{theorem}

\begin{proof}
(i) $\Rightarrow$ (ii): One says that  $f\in \Baire $ is \emph{dominant} if $g \le^*f$ for each computable function $g$. Suppose $A\geT f$ for some dominant function $f$. Fix an  effective list $P_0,P_1,\cdots$ of all $\Pi^0_1$-classes, and for each $i$ let $h_i(n)$ be  the least stage $s>n$ such that for every $\sigma\in 2^{n}$ there exists some $\tau\succeq\sigma$ such that $[\tau]\cap P_{i,s}=\emptyset$. Then each $h_i$ is partial computable, and $h_i$ is total iff $P_i$ is nowhere dense.

Define the closed set
\[F_{i,n} = \begin{cases} \emptyset, &\mbox{if } (\exists x>n)~(h_i(x)\uparrow ~\vee ~h_i(x) > f(x)),\\
P_i, & \mbox{if } \text{otherwise}. \end{cases}
\]
Note that the predicate ``$(\exists x>n)~(h_i(x)\uparrow ~\vee ~h_i(x) > f(x))$'' is $\Sigma^0_1(A)$ and so $\{F_{i,n}\}$ is an $A$-computable sequence of $\Pi^0_1(A)$-classes, which are all nowhere dense. Now fix $i$ such that $P_i$ is nowhere dense, i.e. $h_i$ is total. Since $f$ is dominant let $n$ be such that $f(x)>h_i(x)$ for every $x>n$. We have $F_{i,n}=P_i$.

(ii) $\Rightarrow$ (iii): For this easy direction, note that an oracle $A$ is meager engulfing iff there exists an $A$-effectively meager $F_\sigma(A)$-class containing all nowhere dense $\Pi^0_1$-classes (i.e., we may replace being effectively meager by being nowhere dense).

(iii) $\Rightarrow$ (i): Suppose $A$ is meager engulfing. Let $\bigcup_i G_i$ be an  $A$-effectively meager $F_\sigma (A)$-class in the sense of Subsection~\ref{ss:Schnorr meager} containing all nowhere dense $\Pi^0_1$-classes. Let $f(n)$ be defined by
\[f(n)=\max_{i<n,\sigma\in ^n 2}\left\{|\tau|\mid \tau\supseteq\sigma\text{ is the first found such that }[\tau]\cap G_{i}=\emptyset\right\}\]
Then $f$ is total (as each $G_i$ is nowhere dense) and $f\leT A$.

We claim that the function $f$ is dominant. Suppose not. Let $h$ be a computable function and let  the increasing  sequence $\{x_n\}$ be such that $h(x_n)>f(x_n)$ for every $n$. We will define a nowhere dense $\Pi^0_1$-class $P$ such that $P\not\subseteq \bigcup_i G_i$. Given string $\tau$ and $n\in\omega$ we say that $\tau$ is $n$-good if $\tau$ is of the form $0^k*1*\tau'*1$ where $k\in\omega$ and $\tau'$ can be any binary string of length $h(n+k+1)$. An infinite binary string $X$ is good if there are strings $\sigma_0,\sigma_1,\cdots$ such that $X=\sigma_0*\sigma_1*\cdots$, $\sigma_0$ is $0$-good, and for each $i$, $\sigma_{i+1}$ is $(|\sigma_0|+|\sigma_1|+\cdots+|\sigma_i|)$-good. Now let $P$ be the set of all infinite binary strings $X$ which are good. Clearly the complement of $P$ is open and is generated by a computable set of basic neighbourhoods. It is also clear that $P$ is nowhere dense, since if $\sigma_0*\cdots*\sigma_i$ is an initial segment of a good path, then the string $\sigma_0*\cdots*\sigma_i*1*\tau'*0$ is not extendible in $P$ for any $\tau'$ of length $h(|\sigma_0|+|\sigma_1|+\cdots+|\sigma_i|+1)$.


Now we use the sequence $\{x_n\}$ to build a path $X\in P$ such that $X\not\in G_i$ for every $i$. (We remark that since the sequence $\{x_n\}$ is $A$-computable, the construction below will produce an $A$-computable path $X$). We inductively define strings $\eta_0,\eta_1,\cdots$ such that $\eta_0$ is $0$-good and $\eta_{i+1}$ is $(|\eta_0|+|\eta_1|+\cdots+|\eta_i|)$-good for every $i$. At the end we take $X=\eta_0*\eta_1*\cdots$ and so $X\in P$. We will also explicitly ensure that $X\not\in G_i$ for any $i$.

\emph{Construction of $X$}. Suppose that $\eta_0,\cdots,\eta_i$ have been defined satisfying the above, so that $\eta_0*\cdots*\eta_i$ is extendible in $P$. Find the least $j$ such that $x_j>|\eta_0|+\cdots+|\eta_i|$, and take $k=x_j-|\eta_0*\cdots*\eta_i|-1\geq 0$. Let $\tau'$ be any string of length strictly equal to $h(x_j)$ such that $\left[\eta_0*\cdots*\eta_i*0^k*1*\tau'\right]\cap G_{i+1}=\emptyset$. This $\tau'$ exists because $h(x_j)>f(x_j)$ and can be found $A$-computably. Now take $\eta_{i+1}=0^k*1*\tau'*1$, which will be $(|\eta_0*\cdots*\eta_i|)$-good.

It is straightforward to check the construction that $X=\eta_0*\eta_1*\cdots$ is good and that for every $i$, $\left[\eta_0*\cdots*\eta_i\right]\cap G_i=\emptyset$. Hence $P$ is a nowhere dense set not contained in $\bigcup_i G_i$, a contradiction.
\end{proof}

\begin{theorem}[Rupprecht, {Ref.~\citenum{Rup:The} Corollary~VI.12}]\label{thm:WME} The following are equivalent for an oracle $A$.
\begin{itemize}
\item[(i)] $A$ is in $\calB(\neq^*)$, that is, there is some $f \leT A$ that eventually disagrees with each computable function.
\item[(ii)] $A$ is weakly meager engulfing. \end{itemize}
\end{theorem}

\begin{proof} (i) $\Rightarrow$ (ii): Let $f \leT A$ be given. Then the classes  $\calC_m = \{ x :
(\forall n \geq m) (x(n) \neq f(n)) \}$ are uniformly $\Pi^0_1 (A)$ (in the Baire space) and all $\calC_m$ are nowhere
dense. Thus the set $D = \bigcup_m \calC_m$ is effectively meager and contains all computable functions in $\Baire$.

(ii) $\Rightarrow$ (i): We now work in the Cantor space. Suppose that (i) fails. Let $V=\bigcup_e V_e$ be a meager set relative to $A$. For each $n$ we let $\sigma_n$ be the first string found such that $V_i\cap[\tau*\sigma_n]=\emptyset$ for every $\tau\in 2^n$ and $i\leq n$. Since the degree of $A$ is not high, let $p$ be a strictly increasing computable function not dominated by the function $n+|\sigma_n|\leq_T A$, with $p(0)>0$. Call a pair $(m,\tau)$ good with respect to $n$ if $p^{n}(0)\leq m < p^{n+1}(0)$ and $m+|\tau|<p^{n+2}(0)$. Here $p^0(0)=0$ and $p^{n+1}(0)=p(p^n(0))$. Call two good pairs $(m_0,\tau_0)$ and $(m_1,\tau_1)$ disjoint if $(m_i,\tau_i)$ is good with respect to $n_i$ and $n_0+2\leq n_1$ or $n_1+2\leq n_0$. By the choice of $p$ there are infinitely many numbers $m$ such that $(m,\sigma_m)$ is good and pairwise disjoint from each other.

 Define $f(n)$ to code the natural sequence $(m^n_0, \sigma_{m^n_0}),\cdots,(m^n_{3n},\sigma_{m^n_{3n}})$ such that each pair in the sequence is good and the pairs are pairwise disjoint (from each other and from all previous pairs coded by $f(0),\cdots,f(n-1)$). Let $h$ be a computable function infinitely often equal to $f$. We may assume that each $h(n)$ codes a sequence of the form $(t^n_0, \tau^n_0),\cdots,(t^n_{3n}, \tau^n_{3n})$ where each pair in the sequence is good for some number larger than $n$, and that the pairs in $h(n)$ are pairwise disjoint. (Unfortunately we cannot assume that $(t^n_{i}, \tau^n_{i})$ and $(t^m_{j}, \tau^m_{j})$ are disjoint if $n\neq m$.) This can be checked computably and if $h(n)$ is not of the correct form then certainly $h(n)\neq f(n)$ and in this case we can redefine it in any way we want.

We  define the computable real $\alpha$ by the following. We first pick the pair $(t^0_0,\tau^0_0)$. Assume that we have picked a pair from $h(i)$ for each $i<n$, and assume that the $n$ pairs we picked are pairwise disjoint. From the sequence coded by $h(n)$ there are $3n+1$ pairs to pick from, so we can always find a pair from $h(n)$ which is disjoint from the $n$ pairs previously picked. Now define $\alpha$ to consist of all the pairs we picked, i.e. if $(t,\tau)$ is picked then we define $\alpha\supset(\alpha\upharpoonright{t})*\tau$, and fill in $1$ in all the other positions. This $\alpha$ is computable because for each $i$, $h(i)$ must code pairs which are good for some $n>i$. It is easily  checked that $\alpha$ is not in $\bigcup_e V_e$.
\end{proof}


\subsection{Allowed cuts of the diagram} \label{ss:separate}

A {\em  cut} of Cicho\'n's diagram is a partition    of the set of nodes into two nonempty sets  $L,R$ such that edges leaving $L$ go to the right or  upward. For any   cut not contradicting the ternary  relationships $\add(\calM) =\min(\frb,\cov(\calM))$ and $\cof(\calM) =\max(\frd,\non(\calM))$, it  is consistent with $ZFC$  to  assign the cardinal $\aleph_1$ to all nodes in $L$, and $\aleph_2$ to all nodes in $R$.
See Sections 7.5 and 7.6 of Ref.~\citenum{Bartoszynski.Judah:book} for all of the models.

%

We expect the same to be true for the computability-theoretic diagram: for any allowed cut $L,R$, there is a  degree
satisfying all the properties in $R$, and none in $L$.  There are still several
open questions. Of course, since there are more equivalences on the computability theoretic side, there
are fewer possible combinations.

Here is a list of   possible combinations.

\begin{enumerate}
\item \emph{There is a set $A \geT $ a Schnorr random which is not high yet of hyperimmune degree}.  (This corresponds to the cut where $L$ only contains the property of highness.) The fact  follows by considering a low random real, whose existence is guaranteed by the low basis theorem --- see for example Ref.~\citenum{Nies:book} Theorem~1.8.37.

\item \emph{There is a set  $A \geT $ a Schnorr random which is weakly Schnorr engulfing and of hyperimmune-free degree}. (This corresponds to the cut where $L$   contains the properties of highness and of having hyperimmune degree.)  The fact  follows by taking a set $A$ of hyperimmune-free PA degree (see e.g.~1.8.32 and 1.8.42 of Ref.~\citenum{Nies:book} for the existence of such $A$).
This set is weakly Schnorr engulfing (by Theorem \ref{thm:PAweaklySE} below) and also computes a Schnorr random (by the Scott basis theorem, see e.g. ~Ref.~\citenum{Downey.Hirschfeldt:book}: 2.21.2 or~Ref.~\citenum{Nies:book}: 4.3.2).

Note that Rupprecht shows in Corollary~27 of Ref.~\citenum{Rupprecht:10} that if $B$ is a hyperimmune-free Schnorr
random, then $B$ is not weakly Schnorr engulfing; this shows the example $A$ itself cannot be Schnorr random.
Intuitively, the example $A$ is ``larger than the Schnorr random". One way to view this (from the forcing-theoretic point of view)
is as a two-step iteration: first add a Schnorr random $B$ of hyperimmune-free degree and then a hyperimmune-free
weakly Schnorr engulfing $A$, see e.g.~Ref.~\citenum{Rupprecht:10} Proposition~18 or Theorem~19.

\item \emph{There is a set $A$  which is not weakly Schnorr engulfing  and computes a Schnorr random}. This follows by taking $A$ to be
Schnorr random of hyperimmune-free degree (which is possible by the basis theorem for computably dominated sets, 1.8.42 of Ref.~\citenum{Nies:book}), see~Ref.~\citenum{Rupprecht:10} Corollary 27.

\item \emph{There  is a weakly meager engulfing  set $A$ of hyperimmune degree which computes no Schnorr random.}
Miller~\cite{Miller:hausdorff} (see also~Ref.~\citenum{Downey.Hirschfeldt:book}, 13.8) proved that there is a $\Delta^0_2$ set $A$
which has effective Hausdorff dimension ${1 \over 2}$ and does not compute a real of higher dimension.
Such $A$ necessarily is DNR (Ref.~\citenum{Downey.Hirschfeldt:book}, 13.7.6) and does not compute a Martin-L\"of
random. Rupprecht in Theorem VI.19 of Ref.~\citenum{Rup:The} showed that $A$ is low$_2$. In particular, $A$ is not
high and thus does not compute a Schnorr random either 
(Ref.~\citenum{Nies:book}, 3.5.13). Since $A $ is $\Delta^0_2$, it is  of hyperimmune degree.

\item \emph{There is a   set of hyperimmune degree which is not weakly meager engulfing.} Any non recursive low r.e.\ set will suffice. By Arslanov's completeness criterion (Ref.~\citenum{Nies:book}, 4.1.11), such a set cannot be DNR.
Rupprecht provides another source of examples in Theorem~VI.4
of his thesis\cite{Rup:The},
showing that no 2-generic real is weakly meager engulfing.

\item \emph{There is a weakly Schnorr engulfing set which is low for weak 1-genericity}. This is Ref.~\citenum{Rupprecht:10} Theorem~19.
\end{enumerate}

\begin{question} \label{qu:what's left}

\begin{enumerate}  \setcounter{enumi}{6}

\item Is there a weakly meager engulfing set $A$ which does not  compute a Schnorr random,  is of hyperimmune-free degree, and   is weakly Schnorr engulfing?  ($L$ consists of highness, computing a Schnorr random, and being of hyperimmune degree.)


\item Is there a weakly meager engulfing set which neither computes a Schnorr random nor is weakly Schnorr engulfing?  ($R$ consists of weakly meager engulfing, not low for weak 1 genericity, and not low for Schnorr tests.)

\item Is there a set which is not low for Schnorr tests, is low for weak 1-genericity and not
   weakly Schnorr engulfing? ($R$ only  contains the property of being not low for Schnorr tests.)

\end{enumerate}
\end{question}

Kumabe and Lewis~\cite{Kumabe.Lewis:09} show that there is a set $A$ of minimal Turing degree which is DNR. It was observed in Theorems~1.1 and 3.1 of Ref.~\citenum{Jockusch_lewis} that this construction also makes $A$ of hyperimmune-free degree. Since $A$ cannot compute a Schnorr random set, there is a positive answer to   (7)  or (8)  in Question~\ref{qu:what's left}.   We note that a positive answer to (9)    would refute Conjecture~31 of Ref.~\citenum{Rupprecht:10}. Kjos-Hanssen and Stephan have announced an affirmative answer to further questions stated above.




%
%

\section{Other cardinal characteristics and their analogs}


\subsection{Kurtz randomness and closed measure zero sets} A closed measure zero set necessarily is nowhere
dense. Thus the $\sigma$-ideal $\calE$ generated by closed measure zero sets is contained in both
$\calM$ and $\calN$; in fact, it is properly contained in $\calM \cap \calN$. In a combinatorially intricate
work~\cite{Bartoszynski.Shelah:92}, Bartoszy\'nski and Shelah computed the cardinal characteristics of $\calE$
(see alternatively Section~2.6 of Ref.~\citenum{Bartoszynski.Judah:book}). Main results are:
\begin{itemize}
\item[(A)] $\add(\calE) = \add(\calM)$ and, dually, $\cof(\calE) = \cof(\calM)$;
\item[(B)] $\add(\calE,\calN) = \cov(\calM)$ and, dually, $\cof(\calE,\calN) = \non(\calM)$;
\item[(C)] if $\cov(\calM) = \frd$, then $\cov(\calE) = \max \{ \cov (\calM), \cov (\calN) \}$; \\
   dually, if $\non(\calM) = \frb$, then $\non(\calE) = \min \{ \non (\calM), \non (\calN) \}$.
\end{itemize}
Here, for two ideals $\calI \subseteq \calJ$, $\add (\calI,\calJ)$ denotes the least size of a family of
sets in $\calI$ whose union does not belong to $\calJ$. Similarly, $\cof(\calI,\calJ)$ is the smallest
cardinality of a subfamily $\calF$ of $\calJ$ such that all members of $\calI$ are contained in a
set from $\calF$.

The notion corresponding to $\calE$ and its characteristics on the computability theory side is {\em Kurtz randomness}:
a {\em Kurtz test} is an effective sequence $(G_m)$ of clopen sets such that each $G_m$ has measure
at most $2^{-m}$. The corresponding null $\Pi^0_1$ class $\bigcap_m G_m$ is called a {\em Kurtz
null set}. It is well-known and easy to see that the definition of 
Kurtz null set is unchanged
if we additionally assume $G_{m+1} \sub G_m$ for all $m$.
A real $A$ is {\em Kurtz random} if it passes all Kurtz tests, i.e., $A$ avoids $\bigcap_m G_m$
for all Kurtz tests $(G_m)$.  An oracle $A$ is {\em low for Kurtz tests} if for every Kurtz test $ (G_m)$ relative to $A$, there is a Kurtz test $ (L_k)$  such that $\bigcap_m G_m \sub \bigcap_k L_k$. $A$ is {\em low for Kurtz randomness}  if every Kurtz random
is Kurtz random relative to $A$. Finally, $A$ is {\em low for Schnorr-Kurtz} if every
Schnorr random is Kurtz random relative to $A$.

Greenberg and Miller (Ref.~\citenum{Greenberg.Miller:09}, Theorem~1.1) proved that a set is low for Kurtz tests iff  it is low for weak 1-genericity. This is the computability theoretic analogue of the
dual form $\cof(\calE) = \cof(\calM)$ of (A) above. (They also observed that low for Kurtz
randomness is the same as low for Kurtz tests.) Furthermore, they showed (Ref.~\citenum{Greenberg.Miller:09}, Corollary~1.3) that
a set is low for Schnorr-Kurtz iff it is neither DNR nor high. Thus, by Ref.~\citenum{Kjos.ea:2011} Theorem~5.1
and our Theorem~\ref{thm:WME}, a set is not low for Schnorr-Kurtz iff it is weakly meager
engulfing. This corresponds to $\cof(\calE,\calN) = \non(\calM)$ in (B) above.
Finally, it is well-known (see e.g. Ref.~\citenum{Nies:book} Proposition~3.6.4) that a Kurtz random either is of hyperimmune degree
(and thus contains a weakly 1-generic) or is already Schnorr random, an analogue of
the first part of (C) above. (Note that the antecedent $\cov(\calM) = \frd$ is true in computability theory:
up to Turing degree, weakly 1-generic  = hyperimmune.)

We now look into the computability theoretic aspect of the dual results of the Bartoszy\'nski-Shelah
theorems. To this end, say that an oracle $A$ is {\em Kurtz engulfing} if there is an $A$-computable
sequence $\{ G^i \} = \{ (G^i_m ) \}$ of Kurtz tests relative to $A$ such that each Kurtz null set $\bigcap_m L_m$
is contained in some $\bigcap_m G^i_m$. $A$ is {\em weakly Kurtz engulfing} if there is such $\{ G^i \} = \{ (G^i_m ) \}$
such that $\bigcup_i \bigcap_m G^i_m$ contains all computable reals. Finally, $A$ is {\em Schnorr-Kurtz engulfing} if the union of all
Kurtz null sets is a Schnorr null set relative to $A$. Then we obtain:

\begin{theorem} The following are equivalent for an oracle $A$.   \label{thm:Kurtz1}
\begin{itemize}
\item[(i)] $A$ is high.
\item[(ii)] There is a Kurtz test $(G_m)$ relative to $A$ such that for all Kurtz tests $(L_m)$ and almost all $m$,
   $L_{2m} \cup L_{2m + 1} \subseteq G_m$.
\item[(iii)] There is a Kurtz test $(G_m)$ relative to $A$ such that for all Kurtz tests $(L_m)$ there
   is an $m_0$ with $\bigcap_{m \geq m_0} L_m \subseteq \bigcap_{m \geq m_0} G_m$.
\item[(iv)] $A$ is Kurtz engulfing. \end{itemize}
\end{theorem}

\begin{proof}
(i) $\Rightarrow$ (ii): We use Theorem 6 of Ref.~\citenum{Rupprecht:10}: since $A$ is high, there is a trace $\sigma \leT A$ tracing
all computable functions. Fix a computable coding $c(m,i)$ of all basic clopen sets in $2^\omega$ such that for fixed $m$,
$c (m, \cdot)$ lists all basic clopen sets of measure $\leq 2^{-m}$.
Let $G_m = \bigcup \{ c(2m,i) : i \in \sigma (2m) \} \cup \bigcup \{ c(2m + 1 , i) :
i \in \sigma (2m + 1) \}$. Then $\lambda G_m \leq 2m \cdot 2^{-2m} + (2m + 1) \cdot 2^{- (2m +1)} < 2^{-m}$ for $m \geq 4$. So,
changing finitely many $G_m$ if necessary, we may think of $(G_m)$ as a Kurtz test relative to $A$. Now, given a
Kurtz test $L = (L_m)$, define a function $f = f_L$ by $f (m) = \min\{ i : L_m = c(m,i) \}$. Clearly $f$ is computable
and $L_m = c(m,f(m))$. Thus $f \in^* \sigma$.
This means that $L_{2m} \cup L_{2m +1} = c(2m, f (2m)) \cup c(2m +1, f (2m + 1)) \sub G_m$ for almost all $m$,
as required.

(ii) $\Rightarrow$ (iii) $\Rightarrow$ (iv): Trivial.

(iv) $\Rightarrow$ (i): Assume $A$ is not high and $\{ G^i \} = \{ (G^i_m ) \}$ is an $A$-computable sequence of
Kurtz tests relative to $A$. This means that there are sequences $(k^i_m)$ and $(\sigma^i_{m,j})$ computable in $A$
such that each $G^i_m$ is of the form $G^i_m = \bigcup_{j < k^i_m} [\sigma^i_{m,j}]$. It is easy to see that
each $\sigma^i_{m,j}$ belongs to $2^k$ for some $k \geq m$. It suffices to find a Kurtz test $(L_m)$ and
a set $X \leT A$ such that $X \in \bigcap_m L_m$ yet $X \notin \bigcup_i \bigcap_m G^i_m$, that is,
for each $i$ there is $m$ with $X \notin G^i_m$.

Define a function $f \leT A$ as follows.
\[\begin{array}{rcl}
f(0) & = & 0 \\
f(n+1) & = & \min \{ k : \; \forall m \leq f(n) + (n+1)^2 \; \forall i < n+ 1 \; \forall j < k^i_m \; (\sigma^i_{m,j} \in 2^{\leq k}) \} \\
\end{array}\]
Then clearly $f(n+1) \geq f(n) + (n+1)^2$; in particular, $f$ is strictly increasing. Since $A$ is not high, there is a
computable function $h$ with $h \not\leq^* f$. We may assume that $h$ is strictly increasing as well. Let $J_m$
be the interval $[h(m) , h(m) + m)$ of length $m$. Put $L_m = \{ X : X \restriction J_m \equiv 0 \}$. Then $\lambda
L_m = 2^{-m}$ and therefore $(L_m)$ is a Kurtz test. We produce the required $X \in \bigcap_m L_m \setminus
\bigcup_i \bigcap_m G^i_m$ by recursively defining $X \restriction f(n)$ such that $X \restriction (f(n) \cap J_m) \equiv 0$
for all $m$.

$X \restriction f(0) = X \restriction 0$ is the trivial sequence. Assume $X \restriction f(n)$ has been defined.
If $h(n+1) \leq f(n+1)$ simply extend $X \restriction f(n)$ to $X \restriction f(n+1)$ such that $X \restriction (f(n+1) \cap J_m)
\equiv 0$ for all $m$. So suppose $h(n+1) > f(n+1)$. Note that, when extending $X$ from $f(n)$ to $f(n+1)$, there
are at most $| \bigcup_{m \leq n} J_m | \leq n^2$ many places where $X$ necessarily has to assume the value $0$.
Hence the measure of the set of possible extensions of $X$ to $f(n+1)$ is $\geq 2^{- (f(n) + n^2)}$. On the other hand,
\[\lambda \left(\bigcup \{ G^i_{f(n) + (n+1)^2}  : {i \leq n} \} \right) \leq (n+1) 2^{- (f(n) + (n+1)^2)} < 2^{- (f(n) + n^2)}\] This means that
we can extend $X \restriction f(n)$ to $X \restriction f(n+1)$ such that $X \restriction (f(n + 1) \cap J_m) \equiv 0$ for all $m$ and
$[X \restriction f(n+1)] \cap G^i_{f(n) + (n+1)^2} = \emptyset$ for all $i \leq n$. Since $h(n+1) > f(n+1)$ for infinitely many $n$,
$X$ is as required.
\end{proof}

\begin{theorem} The following are equivalent for an oracle $A$.  \label{thm:Kurtz2}
\begin{itemize}
\item[(i)] $A$ is of hyperimmune degree.
\item[(ii)] $A$ is Schnorr-Kurtz engulfing. \end{itemize}
\end{theorem}

\begin{proof}
(i) $\Rightarrow$ (ii): By Proposition~\ref{prop:W1G}, we know that there is a function $f \leT A$
infinitely often equal to all computable reals. As in the previous proof, let $c (n, \cdot)$ be a computable coding of all basic clopen sets
of measure $\leq 2^{-n}$. Then $h (n) = c (n, f(n))$ is a sequence of clopen sets computable in $A$
with $\lambda ( h(n)) \leq 2^{-n}$. Thus, by Ref.~\citenum{Rupprecht:10} Proposition~3,
$N = \bigcap_m \bigcup_{n\geq m} h(n)$ is a Schnorr null set relative to $A$. We need to show
it contains all Kurtz null sets. Let $(G_n)$ be a Kurtz test. Then $G_n = c(n,k(n))$
for some computable function $k$. Hence $k(n) = f(n)$ for infinitely many $n$. Now assume
$X \in \bigcap_n G_n$. Fix any $m$. There is $n \geq m$ with $k(n) = f(n)$. Thus
$X \in G_n = c(n,k(n)) = c (n,f(n)) = h(n) \subseteq \bigcup_{\ell\geq m} h(\ell)$. Unfixing $m$ we
see that $X$ belongs to $N$, as required.

(ii) $\Rightarrow$ (i): Assume $A$ is of hyperimmune-free degree and $N$ is a Schnorr null set relative to $A$.
By Ref.~\citenum{Rupprecht:10} Proposition~3, we may assume that $N = \bigcap_n \bigcup_{m \geq n} E_m$
where $E_m$ is a sequence of clopen sets computable in $A$ with $\lambda E_m \leq 2^{-m}$.
There are sequences $(k_m)$ and $(\sigma_{m,j})$ computable in $A$ such that each $E_m$ is
of the form $E_m = \bigcup_{j < k_m} [\sigma_{m,j}]$. It suffices to
find a Kurtz test $(L_m)$ and a set $X \leT A$ such that $X \in \bigcap_m L_m$ yet $X \notin N$, that is,
there is some $m_0$ with $X \notin \bigcup_{m \geq m_0} E_m$.

We proceed as in the proof of (iv) $\Rightarrow$ (i) of Theorem~\ref{thm:Kurtz1}. Define $f \leT A$ by:
\[\begin{array}{rcl}
f(0) & = & 0 \\
f(n+1) & = & \min \{ k : \; \forall m \leq f(n) + (n+1)^3 \; \forall j < k_m \; (\sigma_{m,j} \in 2^{\leq k}) \} \\
\end{array}\]
Since $A$ is  of hyperimmune-free degree, there is a computable function $h$ with $h \geq^* f$. We may assume
that $h \geq f$ everywhere and that $h$ is strictly increasing. As in the previous proof, let $J_m$
be the interval $[h(m) , h(m) + m)$ of length $m$ and define a Kurtz test $(L_m)$ by $L_m =
\{ X : X \restriction J_m \equiv 0 \}$. Let $m_0 = f(2) + 28$. We produce the required $X \in \bigcap_m L_m \setminus
\bigcup_{m \geq m_0} E_m$ by recursively defining $X \restriction f(n)$ such that
\begin{itemize}
\item[(a)] $X \restriction (f(n) \cap J_m) \equiv 0$ for all $m (< n)$,
\item[(b)] $[X \restriction f(n) ] \cap \bigcup \{ E_\ell : {f(2) + 27 < \ell \leq f(n-1) + n^3} \}= \emptyset$,
\item[(c)] $\lambda ( [X \restriction f(n) ] \cap \bigcup \{ E_\ell : {f(n-1) + n^3 < \ell \leq f(n) + (n+1)^3} \} )
< 2^{- (f(n) + (n-1) n^2)}$ for $n\geq 3$.
\end{itemize}
Clearly, an $X$ satisfying the first two properties for all $n$ is as required. The third property is used to
guarantee the second property along the recursive construction.

$X \restriction f(2)$ is arbitrary satisfying (a). When defining $X \restriction f(3)$, (b) vacuously holds.
Between $f(2)$ and $f(3)$, there are at most $| J_0 \cup J_1 \cup J_2 | = 3$ many places where
$X$ necessarily has to assume the value $0$. Hence the set of extensions of $X$ to $f(3)$ satisfying (a) for $n=3$
has measure at least $2^{- (f(2) + 3)}$. On the other hand,
\[\lambda \left( \bigcup \{ E_\ell : {f(2) + 27 < \ell \leq f(3) + 64} \} \right) < 2^{- (f(2) + 27)}\]
This means that the relative measure of the latter set in the set of possible
extensions is smaller than ${2^{3} \over 2^{3^3} } < 2^{ 3^2 - 3^3 } = 2^{ - 2 \cdot 3^2}$.
Hence there must be one such extension $X \restriction f(3)$ satisfying
$$\lambda \left( [X \restriction f(3) ] \cap \bigcup \{ E_\ell : {f(2) + 27 < \ell \leq f(3) + 64} \} \right)< 2^{- (f(3) + 2 \cdot 3^2)}$$
Thus (c) holds for $n=3$.

More generally, suppose $X\restriction f(n)$ has been defined for $n \geq 3$. Between $f(n)$ and $f(n+1)$, there
are at most $| \bigcup_{m \leq n} J_m | \leq n^2$ many places where $X$ necessarily has to assume the value $0$.
Hence the set of extensions of $X$ to $f(n+1)$ satisfying (a) for $n+1$ has measure at least $2^{- (f(n) + n^2)}$.
By (c) for $n$ we see that the set of extensions satisfying both (a) and (b) for $n+1$ has measure
at least $2^{- (f(n) + n^2 + 1)}$. On the other hand,
\[\lambda \left( \bigcup \{ E_\ell  : {f(n) + (n+1)^3 < \ell \leq f(n+1) + (n+2)^3} \} \right) < 2^{- (f(n) + (n+1)^3)}\]
This means that the relative measure of the latter set in the set of possible
extensions is smaller than ${2^{n^2 +1} \over 2^{(n+1)^3} } < 2^{ (n+1)^2 - (n+1)^3 } = 2^{ - n (n+1)^2}$.
Hence there must be one such extension $X \restriction f(n+1)$ satisfying
\begin{align*}
&\lambda \left( [X \restriction f(n+1) ] \cap \bigcup \{  E_\ell : {f(n) + (n+1)^3 < \ell \leq f(n+1) + (n+2)^3} \} \right)
\\
&< 2^{- (f(n+1) + n (n+1)^2)}.
\end{align*}
Again, this gives (c) for $n+1$.
\end{proof}

\begin{theorem}\label{thm:PAweaklySE} Each PA set is weakly Kurtz engulfing.
\end{theorem}
\begin{proof}
Fix an oracle $A$ computing a $\{0,1\}$-valued function $g$ such that for all $e$, if $J(e) := \varphi_e (e) \downarrow$, then we have $J(e)\neq g(e)$. If $\varphi_e$ is $\{0,1\}$-valued and total then it gives rise naturally to a computable real $X$ where $X(n)=\varphi_e(n)$. Furthermore every computable real can be identified with a total $\varphi_e$ for some $e$.

There is a computable sequence $\{R_e\}$ of pairwise disjoint computable sets such that for every $e$ and $n$, if $\varphi_e(n)\downarrow$ then $J(r_e(n))\downarrow =\varphi_e(n)$, where $r_e(n)$ is the $n^{th}$ element of $R_e$. Now define the $A$-Kurtz test $\{G_k\}$ by
\[G_k=\bigcup_{i<k}[Z_i\upharpoonright 2k],\]
where for every $i$, $Z_i$ is the infinite binary sequence defined by $Z_i(j)=1-g(r_i(j))$. It is then easy to see that every computable real $X$ belongs to $\bigcup_e \bigcap_{k \geq e} G_k$.
\end{proof}

\begin{corollary}
There is a hyperimmune-free weakly Kurtz engulfing degree.
\end{corollary}

\begin{proof}
It is well-known that there is a hyperimmune-free PA degree. (Use the fact that the PA degrees form a $\Pi^0_1$ class and
the basis theorem for computably dominated sets. See Ref.~\citenum{Nies:book} 1.8.32 and 1.8.42.)
See also item (2) in Section~\ref{ss:separate} above.
\end{proof}

We have no characterization of ``weakly Kurtz engulfing" in terms of the other properties and conjecture
there is none. More specifically, we conjecture there is a set both weakly meager engulfing and weakly Schnorr engulfing
that is not weakly Kurtz engulfing. (Note that the antecedent of the second part of (C), $\non(\calM) = \frb$,
is false in computability theory: high is strictly stronger than weakly meager engulfing.)















\subsection{Specker-Eda number and its dual}

The \emph{Specker-Eda number} $\se$ is a cardinal characteristic
introduced by Blass~\cite{Blass:94} in the context of homomorphisms of
abelian groups.
Whilst the original definition would take us too far afield,
there is an equivalent formulation
due to Brendle and Shelah~\cite{Brendle.Shelah:96} that fits in well with
the cardinal characteristics we have already considered.

\begin{definition}
A \emph{partial $g$-slalom} is a function $\varphi:D\to[\omega]^{<\omega}$
with domain $D$ an infinite subset of $\omega$, satisfying
$|\varphi(n)|\leq g(n)$ for all $n\in D$.
\end{definition}
In defining (total) slaloms in Section~\ref{ss:combCichon}, we implicitly
took $g$ to be the identity function.
In the set-theoretic context,
the specific choice of $g$ is in fact irrelevant for our purposes so long as
it goes to infinity; we have given the definition in this way for the sake of
the analogy to come.
As such, for Definition~\ref{sedefn}
we think of $g$ as being fixed.
\begin{definition}\label{sedefn}
The
Specker-Eda number $\se$ is the unbounding number for the relation of being
traced by a partial slalom:
\begin{multline*}
\se =\frb(\in^*_p) =\min\{|\calF|:\calF\subseteq\omega^\omega\land
\forall\text{ partial slalom }\sigma
\exists f\in\calF\\\exists^{\infty}n\in\dom(\sigma)(f(n)\notin\sigma(n))\}.
\end{multline*}
We denote its dual by $\frd(\in^*_p)$:
\begin{multline*}
\frd(\in^*_p)=\min\{|\Phi|:\Phi\text{ is a set of partial slaloms }
\land\\
\forall f\in\omega^\omega
\exists \sigma\in\Phi\forall^{\infty}n\in\dom(\sigma)(f(n)\in\sigma(n))\}.
\end{multline*}
\end{definition}

The cardinal $\se$ sits in the curious part of the diagram in which
cardinals are different set-theoretically but their
computability-theoretic analogues are equivalent
notions: $\add(\calN)\leq\se\leq\add(\calM)$, and each of these inequalities
may be strict --- see Ref.~\citenum{Brendle.Shelah:96} Corollary~(b).

The fact that we are considering partial slaloms raises the consideration of
partial computable rather than just computable traces.  For the analogue
of $\se$, we shall see that the choice is immaterial.

\begin{definition}\label{partialtrace}
Given an infinite computably enumerable subset $D$ of $\omega$ and a partial computable function
$g$ dominating the identity with $D \subseteq \dom (g)$,
a \emph{$(D,g)$-trace} is a partial computable function $\sigma$ from
$D$ to $[\omega]^{<\omega}$ such that for all $n\in D$, $|\sigma(n)|<g(n)$.
A \emph{partial trace} is a $(D,g)$-trace for some $D$ and $g$.
We say that a $(D,g)$-trace $\sigma$ \emph{traces} $f:\omega\to\omega$
if $f(n)\in\sigma(n)$ for all but finitely many $n\in D$.
\end{definition}

The computability-theoretic analogue of $\se$ is the property that
$A$ computes a partial trace tracing every computable function. This property coincides with being high.

\begin{theorem}
The following are equivalent for any oracle $A$.
\renewcommand\theenumi{\roman{enumi}}
\begin{enumerate}
\item\label{setotal} $A$ computes a trace tracing every computable function (i.e. $A$ is high).
\item\label{separtial} $A$ computes a partial trace
tracing every computable function.
\end{enumerate}
\renewcommand\theenumi{\arabic{enumi}}
\end{theorem}
\begin{proof}
(\ref{setotal}$)\Rightarrow($\ref{separtial}) is trivial.
For the reverse direction, it suffices by
Theorem~6 of Ref.~\citenum{Rupprecht:10} 
to show that (\ref{separtial}) implies $A$ is high.
So suppose $\sigma$ is a $(D,g)$-trace computed by $A$, that is,
$D$ is c.e. in $A$ and both $g$ and $\sigma$ are partially computable in $A$, tracing every computable
function; we wish to show that $A$ computes a function eventually
dominating every computable function.
Fix a computable enumeration of $D$, and for each $m\in\omega$ let
$n_m$ be
the first natural number greater than or equal to $m$ that appears in
this enumeration of $D$.
We define $h:\omega\to\omega$ by $h(m)=\max(\sigma(n_m))$.
Now, let $f$ be a computable function from $\omega$ to $\omega$;
without loss of generality
we may assume that $f$ is non-decreasing.
There is some $m_0\in\omega$ such that for all $n\in D\smallsetminus m_0$,
$f(n)\in\sigma(n)$.  In particular, for all $m\geq m_0$,
$f(m)\leq f(n_m)\in\sigma(n_m)$, and so $f(m)\leq h(m)$.
\end{proof}

We turn now to the analogue of $\frd(\in^*_p)$: the property that $A$
computes a function not traceable by any partial trace.
With ``partial trace'' defined as in Definition~\ref{partialtrace},
this property coincides with being of hyperimmune degree, and the proof is straightforward.


\begin{theorem}\label{thm:partialcomptrace}
The following are equivalent for any oracle $A$.
\begin{enumerate}
\item[(i)] $A$ is of hyperimmune degree.
\item[(ii)] $A$ computes a function not traceable by any partial trace.
\item[(iii)] For every partial computable function $g$, there exists $f\leT A$ such that $f$ cannot be traced by a $(D,g)$-trace for any c.e.~$D$.
\end{enumerate}
\end{theorem}
\begin{proof}

(i) $\Rightarrow$ (ii): Let $f$ be a function computable from $A$ which is dominated by no computable function. We may assume that $f$ is increasing. If $f$ is traced by a $(D,g)$-trace $\sigma$ then we let $h(n)=\max(\sigma(x_n))$ where $x_n$ is the first number greater than or equal to $n$ that is enumerated in $D$. Then $h+1$ is a computable function dominating $f$, which is impossible.

(ii) $\Rightarrow$ (iii): Trivial.

(iii) $\Rightarrow$ (i): Let $g(\langle e,n\rangle)=\varphi_e(\langle e,n\rangle)+1$, where $\varphi_e$ is the $e^{th}$ partial computable function. Let $f$ be an $A$-computable function that cannot be traced by any $(D,g)$-trace. Then we claim that $f$ cannot be dominated by a computable function. If $\varphi_e$ dominates $f$ then we take $D=\{\langle e,n\rangle:  n \in \omega \}$ and take $\sigma(x)=\{0,\cdots,\varphi_e(x)\}$. Clearly $\sigma$ is a $(D,g)$ trace tracing $f$, contradiction.
\end{proof}

So the notion of ``partial computably traceable" coincides with being hyperimmune-free if the bound $g$ is allowed to be partial. If $g$ is required to be total, however, we obtain a different notion.
This latter notion is obviously still weaker than being computably traceable, but is now strictly stronger than being hyperimmune-free. In fact:

\begin{proposition}
Suppose that there is a total computable function $g$ such that for every $f\leT A$ there is a $(D,g)$-trace tracing $f$. Then $A$ is of hyperimmune-free degree and not DNR.
\end{proposition}
\begin{proof}
By Theorem \ref{thm:partialcomptrace} we get that $A$ is of hyperimmune-free degree. To show that $A$ is not DNR, let $f$ be $A$-computable. Let $\sigma$ be a $(D,g)$-trace tracing $f$; we may assume that in fact $f(n)\in\sigma(n)$ for all
$n\in D$. Viewing $\sigma$ as a c.e.~trace then allows us
by Ref.~\citenum{Kjos.ea:2011} Theorem~6.2 to deduce that $A$ is not DNR.
\end{proof}

In other words, if we consider
the property that for every total computable $g$, there is
an $f\leT A$ such that no $(D,g)$-trace traces $f$,
then this property lies between ``not low for weak 1-genericity''
and ``not low for Schnorr tests'' in Figure~\ref{Fig:computability_diagram}.
This is analogous to the fact in the set-theoretic setting that
$\cof(\calM)\leq\frd(\in^*_p)\leq\cof(\calN)$.

\subsection{Final comments}
The {\em splitting number $\frs$} is the least size of a subset   $\+ S $ of $\+ P (\omega)$ such that every infinite  set is split by a set in $\+ S$ into two infinite parts. The analog in computability theory is {\em $r$-cohesiveness}: an infinite set $A \subseteq \omega$ is $r$-cohesive if it cannot be split into two infinite parts by a computable set; that is, if $B$ is computable, then either $A \subseteq^* B$ or $A \cap B$ is finite. $A$ is {\em cohesive} if  it cannot be split into two infinite parts by a computably enumerable set. Clearly, cohesive implies $r$-cohesive. ZFC proves that  $\frs \leq \frd, \non (\calE)$ (Ref.~\citenum{Blass:10} Theorems 3.3 and 5.19). On the computability side, $r$-cohesive implies both being of hyperimmune degree, and weakly Kurtz engulfing. On the other hand, $\frs < \add(\calN)$~\cite{Judah.Shelah:88} (see also Ref.~\citenum         {Bartoszynski.Judah:book} Theorem 3.6.21), $\frs > \cov(\calE)$, and $\frs > \frb$ are known to be consistent (the latter two follow from Ref.~\citenum{Blass.Shelah:87}, see the next paragraph). The first has no analog in recursion theory for high implies cohesive~\cite{Jockusch:73}, while the last does by a result of Jockusch and Stephan~\cite{Jockusch.Stephan:93}, who showed that a cohesive set can be non-high. We do not know whether every ($r$-)cohesive degree   computes a Schnorr random.

The dual of the splitting number is the {\em unreaping number} $\frr$, the least size of a subset $\+ S$ of $[\omega]^\omega$ (the infinite subsets of $\omega$) such that every subset of $\omega$ is either almost disjoint from, or almost contains, a member of $\+ S$. To see the duality, consider the relation $R \subseteq [\omega]^\omega \times [\omega]^\omega$ defined by $\langle x,y \rangle \in R$ iff $y$ splits $x$ iff both $x \cap y$ and $x \setminus y$ are infinite. Then $\frs = \frd (R)$ and $\frr = \frb (R)$. The analog of $\frr$ is being of {\em bi-immune degree}, a property introduced by Jockusch\cite{Jockusch:69a}: $A \subseteq \omega$ is {\em bi-immune} if it splits every infinite computable set or, equivalently, if it splits every infinite computably enumerable set, that is, if neither $A$ nor its complement contain an infinite computable (or computably enumerable) set.  In ZFC $\frr \geq \frb, \cov (\calE)$ holds (Ref.~\citenum{Blass:10} Theorems 3.8 and 5.19). Similarly, Kurtz random (and thus also being of hyperimmune degree) implies bi-immune. Jockusch and Lewis~\cite{Jockusch_lewis} recently showed that DNR implies having  bi-immune degree, so that, in fact, not low for weak 1-generic implies bi-immune. This is very different from the situation in set theory, where $\frr < \frs$ (and thus also $\frr < \non(\calE)$ and $\frr < \frd$) is consistent~\cite{Blass.Shelah:87}. We do not know whether there is a weakly Schnorr-engulfing degree that is not bi-immune.

Rupprecht also briefly discusses the analogy between splitting/unreap\-ing and
$r$-cohesive/bi-immune in his thesis\cite[Theorems V.41, 42, 43]{Rup:The}; his treatment is less comprehensive than ours above.

For further open problems see Question~\ref{qu:what's left} above.

\bibliographystyle{plain}

\bibliography{BBNN}
\end{document}